\documentclass{amsart}
\makeatletter
\def\LaTeX{\leavevmode L\raise.42ex
   \hbox{\kern-.3em\size{\sf@size}{0pt}\selectfont A}\kern-.15em\TeX}
\makeatother

\newcommand{\BibTeX}{{\rm B\kern-.05em{\sc
i\kern-.025emb}\kern-.08em\TeX}}
\newtheorem{corollary}{Corollary}[section]
\newtheorem{theorem}{Theorem}[section]
\newtheorem{lemma}{Lemma}

\theoremstyle{definition}
\newtheorem{definition}{Definition}

\numberwithin{equation}{section}

\begin{document}

\title[ Besov vectors and Paley-Wiener vectors in Hilbert spaces ]{Approximation of Besov vectors by Paley-Wiener vectors  in Hilbert spaces}


\maketitle

\begin{center}

\author{Isaac Z. Pesenson }\footnote{ Department of Mathematics, Temple University,
 Philadelphia,
PA 19122; pesenson@temple.edu. The author was supported in
part by the National Geospatial-Intelligence Agency University
Research Initiative (NURI), grant HM1582-08-1-0019. }

\author{Meyer Z. Pesenson }\footnote{ 
CMS Department, California Institute of Technology, MC 305-16, Pasadena, CA 91125; mzp@cms.caltech.edu.
The author was supported in part by the National Geospatial-Intelligence Agency University Research Initiative (NURI), grant HM1582-08-1-0019 and by AFOSR, MURI, Award FA9550-09-1-0643  }

\end{center}

  \begin{abstract}  
We develop an approximation theory in Hilbert spaces that generalizes the classical theory of approximation by entire functions of exponential type. 
The results advance harmonic analysis on manifolds and graphs, thus facilitating data representation, compression, denoising and visualization.
These tasks are of great importance to machine learning, complex data analysis and computer vision.

  \end{abstract}

 \section{Introduction}
 
 One of the main themes in Analysis is correlation between frequency content of a function and its smoothness. In the classical approach the frequency is understood in terms of the Fourier transform (or Fourier series) and smoothness is described in terms of the Sobolev or Lipshitz and Besov norms. For these notions it is well understood \cite{Akh}, \cite{N} that there exists a perfect balance between the rate of approximation by bandlimited functions ( trigonometric polynomials) and smoothness described by Besov norms. For more recent results  of approximations by entire functions of exponential type we refer to \cite{G1}-\cite{G3}.
 
 The classical concepts and result were generalized to  Riemannian manifolds, graphs,  unitary representations of Lie groups and integral transforms in our work \cite{Pes01}-\cite{Pes14}, \cite{KPes}. 
 
  The goal of the present article is to  develop a form of a Harmonic Analysis  which holds true in general Hilbert spaces. In Introduction we formulate  main results obtained
in the paper. The exact definitions of all notions are given in
the text.

We start with a   self-adjoint positive definite operator $L$   in  a Hilbert space $\mathcal{H}$ and consider its positive root
$D=L^{1/2}$. For the operator $D$ one can introduce notion of the
Spectral  Transform $\mathcal{F}_{D}$ which is an
isomorphism between $\mathcal{H}$ and a direct integral of Hilbert
spaces over $\mathbb{R}$.

A Paley-Wiener space $PW_{\omega}(D), \omega>0$, is introduced as
the set of all $f\in \mathcal{H}$ whose image $\mathcal{F}_{D}f$
 has support in $[0,\omega]$. In the
case when $\mathcal{H}=L_{2}(\mathbb{R})^{d} $  and $D$ is a positive square
root from the Laplace operator our definition produces regular
Paley-Wiener spaces of  spherical exponential type.

The domain $\mathcal{D}_{s}, s\in \mathbb{R},$ of the operator $D^{s}, s\in \mathbb{R},$ plays the role of the
Sobolev space and we introduce Besov spaces
$\mathbf{B}^{\alpha}_{2,q}=\mathbf{B}^{\alpha}_{2,q}(D), \alpha>0, 1\leq  q\leq \infty,$
by using Peetre's interpolation $K$-functor \cite{BL}, \cite{BB},
\cite{KPes}, \cite{KPS}, \cite{Tr3}.
\begin{equation}
\mathbf{B}^{\alpha}_{2,q}(D)=\left(\mathcal{H},\mathcal{D}_{r/2}\right)^{K}_{\alpha/r,q},\label{Besov
norm}
\end{equation}
where $r$ can be any natural such that $0<\alpha<r, 1\leq
q<\infty$, or  $0\leq\alpha\leq r,q= \infty$.  It is crucial for
us that Besov norms can be described in terms of a modulus of
continuity constructed in terms  of the Schrodinger group
$e^{itD^{2}}$, wave semigroup  $e^{itD}$, or the heat semigroup
$e^{-tD^{2}}$.
 In what follows the notation $\|\cdot\|$ bellow means
$\|\cdot\|_{\mathcal{H}}$. We introduce a notion of best
approximation
\begin{equation}
\mathcal{E}(f,\omega)=\inf_{g\in
PW_{\omega}(D)}\|f-g\|, \  f\in \mathcal{H}.\label{BA1}
\end{equation}
We also consider the following family of functionals which
describe a rate of decay of the Spectral  transform
$\mathcal{F}_{D}$
\begin{equation}
\mathcal{R}(f,\omega)= \left (\int_{\omega}^{\infty}
\|\mathcal{F}_{D}(f)(\lambda )\|^{2}_{X(\lambda )} dm
 (\lambda ) \right )^{1/2}, \omega>0.\label{DFT}
 \end{equation}
The Plancherel Theorem for $\mathcal{F}_{D}$ implies that every
such functional is exactly the best approximation of $f$ by
Paley-Wiener functions from $PW_{\omega}(D)$:
\begin{equation}
\mathcal{R}(f,\omega)= \mathcal{E}(f,\omega)=\inf_{g\in
PW_{\omega}(D)}\|f-g\|.\label{decayFT}
\end{equation}
Our main results are the following.

\begin{theorem} The norm of the Besov space $\mathbf{B}_{2,q}^{\alpha}(D), \alpha>0, 1\leq q\leq\infty$ is
equivalent to the following norms
\begin{equation}
\|f\| +\left(\int_{0}^{\infty}\left(s^{\alpha}\mathcal{E}(f, s)
\right)^{q}\frac{ds}{s}\right)^{1/q},\label{AN}
\end{equation}
\begin{equation}
\|f\|+\left(\sum_{k=0}^{\infty}\left(a^{k\alpha }\mathcal{E}(f,
a^{k})\right)^{q}\right)^{1/q}, a>1.
\end{equation}

\begin{equation}
\|f\| +\left(\int_{0}^{\infty}\left(s^{\alpha}\mathcal{R}(f,s)
\right)^{q}\frac{ds}{s}\right)^{1/q},
\end{equation}
and
\begin{equation}
\|f\|+\left(\sum_{k=0}^{\infty}\left(a^{k\alpha
}\mathcal{R}(f,a^{k})\right)^{q}\right)^{1/q}, a>1.
\end{equation}

\label{maintheorem1}
\end{theorem}

 \begin{theorem}
 A  vector  $f\in \mathcal{H}$ belongs to $ \mathbf{B}_{2,q}^{\alpha}(D), \alpha>0, 1\leq
q\leq \infty,$ if and only if there exists a sequence of  vectors 
$f_{k}=f_{k}(f)\in PW_{a^{k}}(D),a>1,  k\in \mathbb{N}$ such that
the series $\sum_{k}f_{k}$ converges to $f$ in $\mathcal{H}$ and  the
following inequalities hold for some $c_{1}>0, c_{2}>0$ which are
independent on $f\in \mathbf{B}_{2,q}^{\alpha}(D)$
\begin{equation}
c_{1}\|f\|_{\mathbf{B}_{2,q}^{\alpha}(D)}\leq
\left(\sum_{k=0}^{\infty}\left(a^{k\alpha
}\|f_{k}\|\right)^{q}\right)^{1/q}\leq
c_{2}\|f\|_{\mathbf{B}_{2,q}^{\alpha}(D)}, a>1.\label{normequiv}
\end{equation}

\label{maintheorem2}
\end{theorem}
In the case when $\alpha>0, q=\infty$ one has to make appropriate
modifications in the above formulas.

According to (\ref{decayFT}) the functional
$\mathcal{E}(f,\omega)$ is a measure of decay of the Spectral
 Transform $\mathcal{F}_{D}$ and the Theorems
\ref{maintheorem1} and \ref{maintheorem2} show that Besov spaces
on a manifold $M$  describe  decay of the Spectral 
transform $\mathcal{F}_{D}$ associated with any appropriate
operator $D$.

In the case $\mathcal{H}=L_{2}(\mathbb{R}^{d})$ the Theorems \ref{maintheorem1} and
\ref{maintheorem2} are  classical and can be found in \cite{Akh},
\cite{N} and \cite{Tit}. In the case when $\mathcal{H}$ is $L_{2}$-space on a Riemannian manifold or a graph and $D$ is the square root from the corresponding Laplace operator the Theorems 1 and 2 were proved in our papers \cite{Pes01}-\cite{Pes13}.

\section{Paley-Wiener subspaces generated by a self-adjoint operator in a Hilbert space}

Now we describe Paley-Wiener functions  for a self-adjoint
positive definite  operator $D$ in $\mathcal{H}$. According
to the spectral theory \cite{BS} for any self-adjoint operator $D$
in a Hilbert space $\mathcal{H}$ there exist a direct integral of
Hilbert spaces $X=\int X(\lambda )dm (\lambda )$ and a unitary
operator $\mathcal{F}_{D}$ from $\mathcal{H}$ onto $X$, which
transforms domain of $D^{k}, k\in \mathbb{N},$ onto $X_{k}=\{x\in
X|\lambda ^{k}x\in X \}$ with norm

\begin{equation}
\|x(\lambda)\|_{X_{k}}= \left (\int^{\infty}_{0}
 \lambda^{2k}\|x(\lambda )\|^{2}_{X(\lambda )} dm
 (\lambda ) \right )^{1/2}
 \end{equation}
besides $\mathcal{F}_{D}(D^{k} f)=
 \lambda ^{k} (\mathcal{F}_{D}f), $ if $f$ belongs to the domain of
 $D^{k}$.  As it is known, $X$ is the set of all $m $-measurable
  functions $\lambda \rightarrow x(\lambda )\in X(\lambda ) $, for which the norm

$$\|x\|_{X}=\left ( \int ^{\infty }_{0}\|x(\lambda )\|^{2}_{X(\lambda )}
 dm (\lambda ) \right)^{1/2} $$ is finite.

\begin{definition}
We will say that a vector $f$
from $\mathcal{H}$ belongs to the Paley-Wiener space $PW_{\omega
}(D)$ if the support of the Spectral  transform
$\mathcal{F}_{D}f$ belong to $[0, \omega]$. For a vector $f\in
PW_{\omega }(D)$ the notation $\omega_{f}$ will be used for a
positive number such that $[0, \omega_{f}]$ is the smallest
interval which contains the support of the Spectral 
transform $\mathcal{F}_{D}f$.
\end{definition}

Using the spectral resolution of identity $P_{\lambda}$ we define
the unitary group of operators by the formula
$$
e^{itD}f=\int_{0}^{\infty}e^{it\tau}dP_{\tau}f, f\in \mathcal{H},
t\in \mathbb{R}.
$$

Let us introduce the operator
\begin{equation}
\textbf{R}_{D}^{\omega}f=\frac{\omega}{\pi^{2}}\sum_{k\in\mathbb{Z}}\frac{(-1)^{k-1}}{(k-1/2)^{2}}
e^{i\left(\frac{\pi}{\omega}(k-1/2)\right)D}f, f\in \mathcal{H},
\omega>0.\label{Riesz1}
\end{equation}
 Since  $\left\|e^{it\mathcal{L}}f\right\|=\|f\| $ and
\begin{equation}
\frac{\omega}{\pi^{2}}\sum_{k\in\mathbb{Z}}\frac{1}{(k-1/2)^{2}}=\omega,\label{id}
 \end{equation}
 the series in (\ref{Riesz1}) is convergent and it shows that
 $\textbf{R}_{D}^{\omega}$ is a bounded operator in $\mathcal{H}$
 with the norm $\omega$:
\begin{equation}
 \|\textbf{R}_{D}^{\omega}f\|\leq \omega\|f\|, f\in
\mathcal{H}.\label{Riesznorm}
 \end{equation}

 The next theorem contains generalizations of several results
 from the classical harmonic analysis (in particular  the Paley-Wiener theorem)
  and it  follows essentially
from our  results in \cite{Pes1}, \cite{Pes2},
\cite{Pes12}.

\begin{theorem}
The following statements hold:
\begin{enumerate}

  \item the set $\bigcup _{ \omega >0}PW_{\omega}(D)$ is dense in $\mathcal{H}$;

\item the space $PW_{\omega }(D)$ is a linear closed subspace in
$\mathcal{H}$;

\item a function $f\in \mathcal{H}$ belongs to $PW_{\omega}(D)$ if
and only if it belongs to the  set
$$
\mathcal{D}_{\infty}=\bigcap_{k=1}^{\infty}\mathcal{D}_{k}(D),
$$
and for all $s\in \mathbb{R}_{+}$ the following Bernstein
inequality takes place
\begin{equation}
\|D^{s}f\|\leq \omega^{s}\|f\|;\label{B}
\end{equation}

\item
a vector $f\in \mathcal{H}$ belongs to the space
$PW_{\omega_{f}}(D), 0<\omega_{f}<\infty,$ if and only if $f$
belongs to the set $\mathcal{D}_{\infty}$, the limit
$$
 \lim_{k\rightarrow \infty}\|D^{k}f\|^{1/k}
$$
exists and
\begin{equation}
\lim_{k\rightarrow
\infty}\|D^{k}f\|^{1/k}=\omega_{f}.\label{limitcond}
\end{equation}

\item a vector $f\in \mathcal{H}$ belongs to $PW_{\omega}(D)$ if and
only if $f\in \mathcal{D}_{\infty}$ and  the upper bound
\begin{equation}
\sup _{k\in N }\left(\omega ^{-k}\|D^{k}f\|\right)<\infty
\end{equation}
is finite,

\item a vector  $f\in \mathcal{H}$ belongs to $PW_{\omega}(D)$ if
and only if $f\in \mathcal{D}_{\infty}$ and
\begin{equation}
\underline{\lim}_{k\rightarrow\infty}\|D^{k}f\|^{1/k}=\omega<\infty.
\end{equation}
In this case $\omega=\omega_{f}$.

\item a vector  $f\in \mathcal{H}$ belongs to $ PW_{\omega}(D)$ if
and only if it belongs to the to the set $\mathcal{D}_{\infty}$ and the
following Riesz interpolation formula holds
\begin{equation}
(iD)^{n}f=\left(\textbf{R}_{D}^{\omega}\right)^{n}f, n\in
\mathbb{N}; \label{Rieszn}
\end{equation}

 \item  $f\in PW_{\omega}(D)$ if and only if for every $g\in
\mathcal{H}$ the scalar-valued function of the real variable  $
\left<e^{itD}f,g\right>, t\in \mathbb{R}^{1},$
 is bounded on the real line and has an extension to the complex
plane as an entire function of the exponential type $\omega$;

\item $f\in PW_{\omega}(D)$ if and only if the abstract-valued
function $e^{itD}f$  is bounded on the real line and has an
extension to the complex plane as an entire function of the
exponential type $\omega$;

\item $f\in PW_{\omega}(D)$ if and only if the solution $u(t),
t\in \mathbb{R}^{1}$ of the Cauchy problem for the corresponding
abstract Schrodinger equation
$$
i\frac{\partial u(t)}{\partial t}=Du(t), u(0)=f, i=\sqrt{-1},
$$
 has analytic extension $u(z)$ to the
complex plane $\mathbb{C}$ as an entire function and satisfies the
estimate
$$
\|u(z)\|_{\mathcal{H}}\leq e^{\omega|\Im z|}\|f\|_{\mathcal{H}}.
$$

 \end{enumerate}

 \end{theorem}
\section{Direct and Inverse Approximation Theorems}

Now we are going to use the notion of the best approximation
(\ref{BA1}) to introduce Approximation spaces
$E^{\alpha}_{2,q}(D), 0<\alpha<r,r\in \mathbb{N}, 1\leq q\leq
\infty,$  as spaces for which   the following norm is finite
\begin{equation}
\|f\|_{E^{\alpha}_{2,q}(D)}=\|f\|
+\left(\int_{0}^{\infty}\left(s^{\alpha}\mathcal{E}(f, s)
\right)^{q}\frac{ds}{s}\right)^{1/q},\label{AN}
\end{equation}
where $0<\alpha<r, 1\leq q<\infty,$ or $0\leq \alpha\leq r,
q=\infty.$ It is easy to verify that this norm is equivalent to
the following "discrete" norm
\begin{equation}
\|f\|+\left(\sum_{j\in \mathbb{N}}\left(a^{j\alpha }\mathcal{E}(f,
a^{j})\right)^{q}\right)^{1/q},a>1,\label{DAN}
\end{equation}

The Plancherel Theorem for $\mathcal{F}_{D}$ also gives the
following inequality
\begin{equation}
 \mathcal{E}(f,\omega)\leq \omega^{-k}\left (\int_{\omega}^{\infty}
\|\mathcal{F}_{D}(D^{k}f)(\lambda )\|^{2}_{X(\lambda )} dm
 (\lambda ) \right )^{1/2}\leq \omega^{-k}\|D^{k}f\|.
\end{equation}

In the classical Approximation theory the Direct and Inverse
Theorems give equivalence of the Approximation and Besov spaces.
Our goal is to extend these results to a more general setting.

For any $f\in \mathcal{H}$ we
introduce a difference operator of order $m\in \mathbb{N}$ as
\begin{equation}
\Delta^{m}_{\tau}f=(-1)^{m+1}\sum^{m}_{j=0}(-1)^{j-1}C^{j}_{m}e^{j\tau(iD)}f,
\tau\in \mathbb{R}.\label{Dif}
\end{equation}
and the modulus of continuity is defined as
\begin{equation}
\Omega_{m}(f,s)=\sup_{|\tau|\leq
s}\left\|\Delta^{m}_{\tau}f\right\|\label{ModCont}
\end{equation}

  The following Theorem is a generalization of the
classical Direct Approximation Theorem by entire functions of
exponential type \cite{N}.

\begin{theorem} There exists a constant $C>0$ such that for all
$\omega>0$ and all $f$
\begin{equation}
\mathcal{E}(f,\omega)\leq
\frac{C}{\omega^{k}}\Omega_{m-k}\left(D^{k}f, 1/\omega\right),
0\leq k\leq m.\label{J}
\end{equation}
In particular the following embeddings hold true
\begin{equation}
\mathbf{B}_{2,q}^{\alpha}(D)\subset E^{\alpha}_{q}(D), 1\leq q\leq
\infty.\label{embd}
\end{equation}
\end{theorem}
\begin{proof}

If $h\in L_{1}(\mathbb{R})$ is an entire function of exponential
type $\omega$ then for any $f\in{\mathcal H}$ the vector
$$
g=\int _{-\infty}^{\infty}h(t)e^{itD}fdt
$$
belongs to $PW_{\omega}(D).$ Indeed,  for every real $\tau$ we
have
$$
e^{i\tau
D}g=\int_{-\infty}^{\infty}h(t)e^{i(t+\tau)D}fdt=\int_{-\infty}^{\infty}h(
t-\tau)e^{itD}fdt.
$$
 Using this formula we can extend the abstract function $e^{i\tau D}g$ to the
complex plane as
$$
e^{izD}g=\int_{-\infty}^{\infty}h(t-z)e^{itD}fdt.
$$
 Since by assumption $h\in L_{1}(\mathbb{R})$ is an entire function of exponential
type $\omega$
 we have
$$
\|e^{izD}g\|\leq
\|f\|\int_{-\infty}^{\infty}|h(t-z)|dt\leq\|f\|e^{\omega|z|}
 \int_{-\infty}^{\infty}|h(t)|dt.
 $$
It shows that for every functional $g^{*}\in {\mathcal H}$ the function
$\left<e^{izD}g,g^{*}\right>$ is an entire function and
$$
\left|\left<e^{izD}g,g^{*}\right>\right|\leq
\|g^{*}\|\|f\|e^{\omega|z|}\int_{-\infty}^{\infty}|h(t)|dt.
$$
In other words  $\left<e^{izD}g,g^{*}\right>$ is an entire
function of the exponential type $\omega$
  which is bounded on the real line and  application of the classical
Bernstein theorem gives the following  inequality
$$
\left|\left(\frac{d}{dt}\right)^{k}\left<e^{itD}g,g^{*}\right>\right|
\leq\omega^{k}\sup_{t\in\mathbb{R}}\left|\left<e^{itD}g,g^{*}\right>\right|.
$$
Since
$$
\left(\frac{d}{dt}\right)^{k}\left<e^{itD}g,g^{*}\right>=\left<e^{itD}(iD)^{k}g,g^{*}\right>
$$
we obtain for $t=0$
$$
\left|\left<D^{k}g,g^{*}\right>\right|\leq
\omega^{k}\|g^{*}\|\|f\|\int_{-\infty}^{\infty}|h(\tau)| d\tau.
$$
 Choosing $g^{*}$ such that $\|g^{*}\|=1$ and $\left<D^{k}g,g^{*}\right>=\|D^{k}g\|$
  we obtain the following inequality
$$
\|D^{k}g\|\leq
\omega^{k}\|f\|\int_{-\infty}^{\infty}|h(\tau)|d\tau
$$
which  implies that $g$ belongs to $PW_{\omega}(D)$.

Let
\begin{equation}
h(t)=a\left(\frac{\sin (t/n)}{t}\right)^{n}
\end{equation}
where $n\geq m+3$ is an even integer and
$$
a=\left(\int_{-\infty}^{\infty}\left(\frac{\sin
(t/n)}{t}\right)^{n}dt\right)^{-1}.
$$
With such choice of $a$ and $n$ the function $h$ will have the
 following properties:

\begin{enumerate}

\item $h$ is an even nonnegative entire function of exponential
type one;

\item $h$ belongs to $L_{1}(\mathbb{R})$ and its
$L_{1}(\mathbb{R})$-norm is $1$;

\item  the integral
\begin{equation}\int_{-\infty}^{\infty}h(t)|t|^{m}dt
\end{equation} is finite.
\end{enumerate}
Consider
 the following vector
\begin{equation}
\mathcal{Q}_{h}^{\omega,m}(f)=\int_{-\infty}^{\infty}
h(t)\left\{(-1)^{m-1}\Delta^{m}_{t/\omega}f+f\right\}dt,\label{id}
 \end{equation}
 where
\begin{equation}
(-1)^{m+1}\Delta^{m}_{s}f=(-1)^{m+1}\sum^{m}_{j=0}(-1)^{j-1}C^{j}_{m}e^{js(iD)}f=
 \sum_{j=1}^{m}b_{j}e^{js(iD)}f-f,\label{id2}
 \end{equation}
 and
 \begin{equation}
 b_{1}+b_{2}+...
+b_{m}=1.
\end{equation}
The formulas (\ref{id}) and (\ref{id2}) imply the following formula
$$
\mathcal{Q}_{h}^{\omega,m}(f)=\int_{-\infty}^{\infty}h(t)\sum_{j=1}^{m}b_{j}e^{j\frac{t}{\omega}(iD)
} fdt=\int_{-\infty}^{\infty}\Phi(t)e^{t(iD)}fdt.
$$
where
$$
\Phi(t)=\sum_{j=1}^{m}b_{j}\left(\frac{\omega}{j}\right)h\left(t\frac{\omega}{j}\right).
$$
Since the function $h(t)$ is of the exponential type one every function
$h(t\omega/j)$ is of the type $\omega/j$. It also shows that
the function $\Phi(t)$ is of the  exponential  type $\omega$ as well.

Now we estimate the error of approximation of $f$ by 
$\mathcal{Q}_{h}^{\omega,m}(f)$. If the modulus of
continuity is defined as

\begin{equation}
\Omega_{m}(f,s)=\sup_{|\tau|\leq
s}\left\|\Delta^{m}_{\tau}f\right\|\label{dif}
\end{equation}
then since by (\ref{id})
$$
f-\mathcal{Q}_{h}^{\omega,m}(f)=
\int_{-\infty}^{\infty}h(t)\Delta^{m}_{t/\omega}fdt
$$
we obtain
$$
\mathcal{E}(f,\omega)\leq\|f-\mathcal{Q}_{h}^{\omega,m}(f)\|\leq
\int_{-\infty}^{\infty}h(t)\left\|\Delta^{m}_{t/\omega}f\right\|dt\leq
\int_{-\infty}^{\infty}h(t)\Omega_{m}\left(f, t/\omega\right)dt.
$$
Now we are going to use the following inequalities
\begin{equation}
\Omega_{m}\left(f, s\right)\leq s^{k}\Omega_{m-k}(D^{k}f,
s)\label{3.14}
\end{equation}

\begin{equation}
\Omega_{m}\left(f, as\right)\leq \left(1+a\right)^{m}\Omega_{m}(f,
s), a\in \mathbb{R}_{+}.\label{3.15}
\end{equation}
The first one follows from the identity
\begin{equation}
\Delta_{t}^{k}f=\left(e^{itD}-I\right)^{k}f=\int_{0}^{t}...\int_{0}^{t}e^{i(\tau_{1}+...\tau_{k})D}
D^{k}fd\tau_{1}...d\tau_{k},\label{id3}
\end{equation}
where $I$ is the identity operator and $k\in \mathbb{N}$.  The
second one follows from the property
$$
\Omega_{1}\left(f, s_{1}+s_{2}\right)\leq \Omega_{1}\left(f,
s_{1}\right)+\Omega_{1}\left(f,s_{2}\right)
$$
which is easy to verify. We can continue our estimation of
$E(f,\omega)$.
$$
\mathcal{E}(f,\omega)\leq
\int_{-\infty}^{\infty}h(t)\Omega_{m}\left(f, t/\omega\right)dt
\leq \frac{\Omega_{m-k}\left(D^{k}f,
1/\omega\right)}{\omega^{k}}\int_{-\infty}^{\infty}h(t)|t|^{k}(1+|t|)^{m-k}dt\leq
$$
$$
\frac{{C}^{h}_{m,k}}{\omega^{k}}\Omega_{m-k}\left(D^{k}f,
1/\omega\right),
$$
where the integral
 $$
C^{h}_{m,k}=\int_{-\infty}^{\infty}h(t)|t|^{k}(1+|t|)^{m-k}dt
$$
 is finite by the choice of $h$. The inequality (\ref{J}) is proved
  and it implies the second part
of the Theorem.
\end{proof}

In fact we proved a little bit more. Namely for the same choice of
the function $h$ the following holds.
\begin{corollary} For any $0\leq k\leq m, k,m\in \mathbb{N,}$
here exists a constant $C^{h}_{ m, k}$ such that for  all
$0<\omega<\infty$ and all $f\in{\mathcal H}$ the following inequality
holds
\begin{equation}
\mathcal{E}(f,\omega)\leq \|\mathcal{Q}_{h}^{\omega,m}(f)-f\|\leq
\frac{C^{h}_{ m, k}}{\omega^{k}}\Omega_{m-k}\left(D^{k}f,
1/\omega\right),\label{J1}
\end{equation}
where
$$
C^{h}_{m,k}=\int_{-\infty}^{\infty}h(t)|t|^{k}(1+|t|)^{m}dt, 0\leq
k\leq m,
$$
and the operator
$$
\mathcal{Q}_{h}^{\omega,m}:{\mathcal H}\rightarrow PW_{\omega}(D)
$$
is defined in (\ref{id}).
\end{corollary}

Next, we are going to obtain the Inverse Approximation Theorem in
the case $q=\infty$.

\begin{lemma}
If there exist $r>\alpha-n>0, \alpha>0, r,n\in \mathbb{N},$ such
that the quantity
\begin{equation}
\mathbf{b}^{\alpha}_{\infty,n,r}(f)=
\sup_{s>0}\left(s^{n-\alpha}\Omega_{r}\left(D^{n}f,s\right)\right)\label{H}
\end{equation}
 is finite, then there exists a constant $A=A(n,r)$ for
 which
\begin{equation}
 \sup_{s>0}s^{\alpha}\mathcal{E}(f,s)\leq
A(n,r)\mathbf{b}^{\alpha}_{\infty,n,r}(f).
\end{equation}
\end{lemma}
\begin{proof}
 Assume that (\ref{H}) holds, then
$$
\Omega_{r}\left(D^{n}f,s\right)\leq
\mathbf{b}^{\alpha}_{\infty,n,r}(f)s^{\alpha-n}
$$
and  (\ref{J1}) implies
\begin{equation}
\mathcal{E}(f,s)\leq
C_{n+r,n}^{h}s^{-n}\mathbf{b}^{\alpha}_{\infty,n,r}(f)
s^{n-\alpha}=
$$
$$
A(n,r)\mathbf{b}^{\alpha}_{\infty,n,r}(f)s^{-\alpha}.
\end{equation}
Lemma is proved.
\end{proof}
\begin{lemma}
If for an $f\in {\mathcal H}$ and for an $\alpha>0$ the following
upper bound is finite
\begin{equation}
 \sup_{s>0}s^{\alpha} \mathcal{E}(f,s)=
T(f,\alpha)<\infty,
\end{equation}
then for every $r>\alpha-n>0, \alpha>0, r,n\in \mathbb{N},$ there
exists a constant $C(\alpha,n,r)$ such that the next inequality
holds
\begin{equation}
\mathbf{b}^{\alpha}_{\infty,n,r}(f) \leq
C(\alpha,n,r)\left(\|f\|+T(f,\alpha)\right).
\end{equation}\label{Nik}

\end{lemma}
\begin{proof}
The assumption implies that for a given $f\in {\mathcal H}$ and a
sequence of numbers $a^{j},a>1, j=0,1,2,...$ one can find a
sequence $g_{j}\in PW_{a^{j}}(D)$ such that
$$
\|f-g_{j}\|\leq T(f,\alpha) a^{-j\alpha}, a>1.
$$
Then for \begin{equation} f_{0}=g_{0}, f_{j}=g_{j}-g_{j-1}\in
PW_{a^{j}}(D),\label{N1}
\end{equation}
 the series
\begin{equation}
f=f_{0}+f_{1}+f_{2}+....\label{N2}
\end{equation}
converges in ${\mathcal H}$. Moreover, we have the following estimates
$$
\|f_{0}\|=\|g_{0}\|\leq\|g_{0}-f\|+\|f\|\leq\|f\|+T(f,\alpha),
$$
$$
\|f_{j}\|\leq \|f-g_{j}\|+\|f-g_{j-1}\|\leq
T(f,\alpha)a^{-j\alpha}+T(f,\alpha)
a^{-(j-1)\alpha}=T(f,\alpha)(1+a^{\alpha})a^{-j\alpha},
$$
which imply the following inequality
\begin{equation}
\|f_{j}\|\leq C(a,\alpha)
a^{-j\alpha}\left(\|f\|+T(f,\alpha)\right), j\in
\mathbb{N}.\label{3.25}
\end{equation}
Since $f_{j}\in PW_{a^{j}}(D)$ we have for any $n\in \mathbb{N}$
\begin{equation}
\|D^{n}f_{j}\|\leq a^{jn}\|f_{j}\|, a>1,\label{3.26}
\end{equation}
we obtain
$$
\|D^{n}f_{j}\|\leq C(a,\alpha)
a^{-j(\alpha-n)}\left(\|f\|+T(f,\alpha)\right)
$$
 which shows that the series
$$
\sum_{j\in \mathbb{N}}D^{n}f_{j}
$$
converges in ${\mathcal H}$ and because the operator $D^{n}$ is closed
the sum $f$ of this series belongs to the domain of $D^{n}$ and
$$
D^{n}f=\sum_{j\in \mathbb{N}}D^{n}f_{j}.
$$
Next, let $F_{j}=D^{n}f_{j}$ then we have  that $
D^{n}f=\sum_{j}F_{j}, $ where $F_{j}\in PW_{a^{j}}(D)$ and
according to (\ref{3.25}) and (\ref{3.26})
\begin{equation}
\|F_{j}\|=\|D^{n}f_{j}\|\leq a^{jn}\|f_{j}\|\leq
C(a,\alpha)a^{-j(\alpha-n)}(\|f\|+T(f,\alpha)).\label{F}
\end{equation}
Pick a positive $t$ and
a natural $N$ such that
\begin{equation}
a^{-N}\leq t<a^{-N+1},a>1,\label{3.28}
\end{equation}
then we obviously have the following formula for any natural $r$
\begin{equation}
\Delta_{t}^{r}D^{n}f=\sum_{j=0}^{N-1}\Delta_{t}^{r}F_{j}+
\sum_{j=N}^{\infty}\Delta_{t}^{r}F_{j},
\end{equation}
where $\Delta_{t}^{r}$ is defined in (\ref{Dif}). Note, that the
Bernstein inequality and the formula (\ref{id3}) imply that if
$f\in PW_{\omega}(D)$, then
\begin{equation}
\|\Delta_{t}^{r}f\|\leq (t\omega)^{r}\|f\|.
\end{equation}
Since (\ref{F}) and (\ref{3.28}) hold  we obtain for $j\leq N-1$
the following inequalities
$$
\|\Delta_{t}^{r}F_{j}\|\leq (a^{j}t)^{r}\|F_{j}\|\leq
C(a,\alpha)(\|f\|+T(f,\alpha))a^{j(n+r-\alpha)-(N-1)r)}, a>1.
$$
These inequalities imply
\begin{equation}
\left\|\sum_{j=0}^{N-1}\Delta_{t}^{r}F_{j}\right\|\leq
C(a,\alpha)(\|f\|+T(f,\alpha))a^{-r(N-1)}\sum_{j=0}^{N-1}a^{(n+r-\alpha)j}=
$$
$$
C(a,\alpha)(\|f\|+T(f,\alpha))a^{-r(N-1)}\frac{1-a^{(n+r-\alpha)N}}{1-a^{(n+r-\alpha)}}\leq
$$
$$
C(a,\alpha,n,r)(\|f\|+T(f,\alpha))t^{\alpha-n}.
\end{equation}
By applying the following inequality
$$
\|\Delta_{t}^{r}F_{j}\|\leq 2^{r}\|F_{j}\|
$$
to terms with $j\geq N$ we can continue our estimation as
\begin{equation}
\left\|\sum_{j=N}^{\infty}\Delta_{t}^{r}F_{j}\right\|\leq
2^{r}C(a,\alpha)(\|f\|+T(f,\alpha))\sum_{j=N}^{\infty}a^{-(\alpha-n)j}=
$$
$$
C(a,\alpha)
2^{r}(\|f\|+T(f,\alpha))a^{-N(\alpha-n)}(1-a^{(n-\alpha)})^{-1}\leq
$$
$$
C(a,\alpha,n,r)(\|f\|+T(f,\alpha))t^{\alpha-n}.
\end{equation}
It gives the following inequality
$$
\|\Delta_{t}^{r}D^{n}f\|\leq
C(a,\alpha,n,r)t^{\alpha-n}(\|f\|+T(f,\alpha)),
$$
from which one has
$$
\Omega_{r}\left(D^{n}f, s\right)\leq
C(a,\alpha,n,r)(\|f\|+T(f,\alpha))s^{\alpha-n}, s>0,
$$
and
$$
\mathbf{b}^{\alpha}_{\infty,n,r}(f)\leq
C(a,\alpha,n,r)(\|f\|+T(f,\alpha)).
$$
The Lemma is proved.
\end{proof}
Our main result concerning spaces
$\mathbf{B}_{2,\infty}^{\alpha}(D), \alpha>0,$ is the following.
\begin{theorem}The norm of the space $\mathbf{B}_{2,\infty}^{\alpha}D), \alpha>0,$ is
equivalent to the following norms
\begin{equation}
\|f\|+\sup_{s>0}\left(s^{\alpha }\mathcal{E}(f,
s)\right),\label{3.33}
\end{equation}
\begin{equation}
\|f\|+\sup_{s>0}\left(s^{\alpha}\mathcal{R}(f,
s))\right),\label{3.34}
\end{equation}
\begin{equation}
\|f\|+\sup_{k\in \mathbb{N}}\left(a^{k\alpha }\mathcal{E}(f,
a^{k})\right),a>1,
\end{equation}
\begin{equation}
\|f\|+\sup_{k\in \mathbb{N}}\left(a^{k\alpha
}\mathcal{R}(f,a^{k}))\right), a>1.\label{3.36}
\end{equation}
Moreover, a vector $f\in \mathcal{H}$ belongs to $
\mathbf{B}_{2,\infty}^{\alpha}(D), \alpha>0,$ if and only if there
exists a sequence of vectors $f_{k}=f_{k}(f)\in PW_{a^{k}}(D),
a>1,$ such that the series $\sum f_{k}$ converges to $f$ in
${\mathcal H}$ and
\begin{equation}
c_{1}\|f\|_{\mathbf{B}_{2,\infty}^{\alpha}(D)}\leq \sup_{k\in
\mathbb{N}}\left(a^{k\alpha }\|f_{k}\|\right)\leq
c_{2}\|f\|_{\mathbf{B}_{2,\infty}^{\alpha}(D)},a>1,
\end{equation}\label{EN}
for certain $c_{1}=c_{1}(D,\alpha), c_{2}=c_{2}(D,\alpha)$
which are independent of $f\in \mathbf{B}_{2,\infty}^{\alpha}(D)$.
\end{theorem}
\begin{proof}
That the norm of $ \mathbf{B}_{2,\infty}^{\alpha}(D), \alpha>0,$
is equivalent to any of the norms (\ref{3.33})-(\ref{3.36})
follows from the last two Lemmas and (\ref{decayFT}).

 Next, if the
norm (\ref{3.33}) is finite then it was shown in the proof of the
last Lemma that there exists a sequence of vectors
$f_{k}=f_{k}(f)\in PW_{a^{k}}(D), a>1,$ such that the series $\sum
f_{k}$ converges to $f$ in ${\mathcal H}$. Moreover, the inequality
(\ref{3.25}) shows existence of constant $c$ which is independent
of $f\in \mathbf{B}_{2,\infty}^{\alpha}(D)$ for which the following inequality holds
$$
\sup_{k\in \mathbb{N}}\left(a^{k\alpha }\|f_{k}\|\right)\leq
c\|f\|_{\mathbf{B}_{2,\infty}^{\alpha}(D)},a>1,
$$
Conversely, let us assume that there exists a sequence of
vectors 
$f_{k}=f_{k}(f)\in PW_{a^{k}}(D), a>1,$ such that the
series $\sum f_{k}$ converges to $f$ in ${\mathcal H}$ and
$$
\sup_{k\in \mathbb{N}}\left(a^{k\alpha }\|f_{k}\|\right)<\infty.
$$
We have
$$
\mathcal{E}(f, a^{N})\leq
\left\|f-\sum_{k=0}^{N-1}f_{k}\right\|=\sum_{k=N}^{\infty}\|f_{k}\|
\leq\sup_{k\in \mathbb{N}}\left(a^{k\alpha
}\|f_{k}\|\right)\sum_{k=N}^{\infty}a^{-\alpha j}\leq
$$
$$C\sup_{k\in \mathbb{N}}\left(a^{k\alpha
}\|f_{k}\|\right)a^{-N\alpha},
$$
or
$$
\sup_{N}a^{N\alpha}\mathcal{E}(f, a^{N})\leq C\sup_{k\in
\mathbb{N}}\left(a^{k\alpha }\|f_{k}\|\right).
$$
Since we also have
$$
\|f\|\leq \sum_{k}\|f_{k}\|\leq \sup_{k\in
\mathbb{N}}\left(a^{k\alpha }\|f_{k}\|\right)\sum_{k}a^{-\alpha
k}, a>1,
$$
the Theorem is proved.
\end{proof}
The Theorems \ref{maintheorem1} and  \ref{maintheorem2} from
Introduction are  extensions of the Theorem \ref{EN} to all indices
$1\leq q\leq \infty$. Their proofs go essentially along the same
lines as the proof of the last Theorem and are omitted.

 \makeatletter
\renewcommand{\@biblabel}[1]{\hfill#1.}\makeatother

   \end{document}